\documentclass[letterpaper, 10 pt, conference]{ieeeconf}  % Comment this line out

\usepackage{amsmath}
\usepackage{amssymb}
\usepackage{mathrsfs}
\usepackage{graphicx}
\usepackage{epstopdf}
\usepackage{balance}
\newtheorem{assumption}{Assumption}
\newtheorem{lemma}{Lemma}

\newtheorem{remark}{Remark}
\newtheorem{theorem}{Theorem}

\IEEEoverridecommandlockouts
\title{\LARGE \bf
Output Feedback Stabilization of Semilinear Parabolic PDEs using Backstepping
}

\author{Agus Hasan
\thanks{The author is with Maersk Mc-Kinney Moller Institute,
        University of Southern Denmark, 5230 Odense, Denmark.
        Email: {\tt\small agha@mmmi.sdu.dk}.}
}

\begin{document}

\maketitle
\thispagestyle{empty}
\pagestyle{empty}

\begin{abstract}
In this paper, we present output feedback boundary stabilization for a class of semilinear parabolic PDEs with a boundary measurement and an actuation located at the same place. The method uses backstepping transformations, where the state and error systems are proved to be locally exponentially stable in the $\mathbb{H}^4$ norm. The stability of the transformed systems are obtained by constructing a strict Lyapunov function. A numerical example using the FitzHugh-Nagumo equation shows the proposed control law stabilizes the system into its equilibrium solution.
\end{abstract}

\section{INTRODUCTION}

Output feedback stabilization of systems modeled by partial differential equations (PDEs) using backstepping is an active research area, see e.g., \cite{A1,A2,A3,Cerpa1,Mail4,Mail5,Mail6,Mail7}. In the infinite-dimensional backstepping method, the boundary feedback control law and the state observer are designed by employing Volterra integral transformations \cite{Krstic1}. The most striking features is both the control gain and the observer gain can be found analytically for many cases (\cite{Krstic2,Vazquez1}).

For nonlinear finite-dimensional systems, the backstepping method has reached its maturity in the last decade \cite{Kr}. Furthermore, it has industrial application, e.g., in oil well drilling application (\cite{Mail2,Mail3}). However, the success has been limited to linear PDEs. In the last few years, the results on the linear backstepping design have been extended to the nonlinear backstepping design with Volterra nonlinearities in \cite{Vazquez2} and \cite{Vazquez3}. In these papers, the nonlinear infinite-dimensional operators of a Volterra-type with infinite sums of integrals in the spatial variable were introduced for stabilization of semilinear parabolic PDEs. Significant results were in the development of control design for cascade of PDEs and nonlinear ODEs (\cite{Mail1,Niko,Delp,Niko1}). Recent progresses in nonlinear backstepping control design for a 2$\times$2 quasilinear hyperbolic system are presented in \cite{Coron} and \cite{Vaz}. These references are central to the development of the output feedback stabilization of the present paper. The idea is to construct a strict Lyapunov function, previously developed in \cite{Coron1} and \cite{Coron2}, which is locally equivalent to the $\mathbb{H}^4$ norm. The present paper is a continuation of a paper on backstepping boundary control of semilinear parabolic PDEs \cite{Hasan} written by the author.

This paper is organized as follows. The output feedback stabilization problem is stated in section II. In section III, we briefly review the output feedback stabilization results for linear parabolic PDEs with Dirichlet's boundary feedback. The contribution of this paper is presented in section IV. Here, we present our main result, which is proven in section V. A numerical example is presented in VI. Finally, section VII contains conclusions and future works.

\section{PROBLEM STATEMENT}

We consider output feedback stabilization for the following semilinear parabolic PDEs:
\begin{eqnarray}
u_t(x,t) &=& u_{xx}(x,t)+f_{NL}(u(x,t)),\label{main1}
\end{eqnarray}
with boundary conditions
\begin{eqnarray}
u(0,t) &=& 0,\label{main2}\\
u(1,t) &=& U(t).\label{main3}
\end{eqnarray}

\begin{assumption}
The nonlinear function $f_{NL}$ is twice continuously differentiable, i.e., $f_{NL}\in \mathbb{C}^2([0,1])$, and has an equilibrium at the origin, i.e., $f_{NL}(0)=0$.
\end{assumption}

The task is to find a feedback control law $U(t)$ using the infinite-dimensional backstepping design to make the origin of system \eqref{main1}-\eqref{main3} locally exponentially stable using only measurements
\begin{eqnarray}
y(t)=u(1,t).
\end{eqnarray}
System \eqref{main1} arises in heat and mass transfer, biology, and ecology. Two prominent examples are the Fisher's equation, used to model the spatial spread of an advantageous allele \cite{Fisher}, and the FitzHugh-Nagumo equation, used to model nerve membrane (\cite{Fitz,Nagumo}).

\section{OUTPUT FEEDBACK STABILIZATION OF LINEAR PARABOLIC SYSTEMS}

Consider the following linear parabolic systems:
\begin{eqnarray}
u_t(x,t) &=& u_{xx}(x,t)+\lambda u(x,t),\label{lins1}
\end{eqnarray}
with boundary conditions
\begin{eqnarray}
u(0,t) &=& 0,\label{lins2}\\
u(1,t) &=& U(t).\label{lins3}
\end{eqnarray}
If we select the control law as:
\begin{eqnarray}
U(t) &=& \int_0^1\! k(1,y)\hat{u}(y,t)\,\mathrm{d}y,\label{cont}
\end{eqnarray}
where $\hat{u}$ is computed from
\begin{eqnarray}
\hat{u}_t(x,t) &=& \hat{u}_{xx}(x,t)+\lambda\hat{u}(x,t)+p(x)\tilde{u}(1,t)\label{obslins1}
\end{eqnarray}
and where the observer gain is given by
\begin{eqnarray}
p(x) = p(x,1)
\end{eqnarray}
with boundary conditions
\begin{eqnarray}
\hat{u}(0,t) &=& 0,\\
\hat{u}(1,t) &=& U(t).\label{obslins3}
\end{eqnarray}
it can be shown that the origin of \eqref{lins1} is exponentially stable, where $k(x,y)$ is solution of the following second order hyperbolic-type PDEs:
\begin{eqnarray}
k_{xx}(x,y) - k_{yy}(x,y) &=& \lambda k(x,y),\label{ker1}\\
k(x,0) &=& 0,\\
k(x,x) &=& -\frac{\lambda}{2}x.\label{ker3}
\end{eqnarray}
Similarly, the kernel $p(x,y)$ is solution of the following kernel equations:
\begin{eqnarray}
p_{xx}(x,y)-p_{yy}(x,y) &=& -\lambda p(x,y)\label{obsker1}\\
p(0,y) &=& 0\\
p(x,x) &=& -\frac{\lambda}{2}x\label{obsker3}
\end{eqnarray}

Both kernel functions evolve in a triangular domain $\mathcal{T}=\{(x,y):0\leq y\leq x\leq1\}$.
\begin{theorem}
Consider systems \eqref{lins1}-\eqref{lins3} and \eqref{obslins1}-\eqref{obslins3} with control law \eqref{cont} and initial conditions $u_0,\hat{u}_0\in\mathbb{L}^2([0,1])$. Then there exists $\mu>0$ and $c>0$, such that:
\begin{eqnarray}
\|u(\cdot,t)\|^2_{\mathbb{L}^2}+\|\hat{u}(\cdot,t)\|^2_{\mathbb{L}^2} \leq ce^{-\mu t}\left(\|u_0\|^2_{\mathbb{L}^2}+\|\hat{u}_0\|^2_{\mathbb{L}^2}\right)
\end{eqnarray}
\end{theorem}

\begin{proof}
Define the observer estimate error $\tilde{u}=u-\hat{u}$. Consider the following Volterra integral transformations:
\begin{eqnarray}
\hat{\gamma}(x,t) &=& \hat{u}(x,t) -\int_0^x\! k(x,y)\hat{u}(y,t)\,\mathrm{d}y\label{trans}\\
\tilde{u}(x,t) &=& \tilde{\gamma}(x,t) - \int_x^1\! p(x,y)\tilde{\gamma}(y,t)\,\mathrm{d}y\label{errtrans}
\end{eqnarray}
It can be proven that, if the kernels verify \eqref{ker1}-\eqref{obsker3}, then $\hat{\gamma}$ and $\tilde{\gamma}$ verify the following equations:
\begin{eqnarray}
\hat{\gamma}_t(x,t) &=& \hat{\gamma}_{xx}(x,t)-\bar{p}_1(x)\tilde{\gamma}_x(1,t)\label{maint1}\\
\hat{\gamma}(0,t) &=& 0\\
\hat{\gamma}(1,t) &=& 0\label{maint3}\\
\tilde{\gamma}_t(x,t) &=& \tilde{\gamma}_{xx}(x,t)\label{lol1}\\
\tilde{\gamma}(0,t) &=& 0\\
\tilde{\gamma}(1,t) &=& 0,\label{lol3}
\end{eqnarray}
where
\begin{eqnarray}
\bar{p}(x) = p(x)-\int_0^x\! k(x,y)p(y)\,\mathrm{d}y.
\end{eqnarray}
Let
\begin{eqnarray}
U(t) = \frac{1}{2}\int_0^1\! \hat{\gamma}(x,t)^2\,\mathrm{d}x + \frac{A}{2}\int_0^1\! \tilde{\gamma}(x,t)^2\,\mathrm{d}x,\label{U}
\end{eqnarray}
where $A$ to be determined letter. The first derivative with respect to $t$ along \eqref{maint1}-\eqref{lol3} is given by:
\begin{eqnarray}
\dot{U}(t) &=& -\int_0^1\! \hat{\gamma}_x(x,t)^2 \,\mathrm{d}x -A\int_0^1\! \tilde{\gamma}_x(x,t)^2 \,\mathrm{d}x \nonumber\\
&&- \int_0^1\! \hat{\gamma}(x,t)\bar{p}_1(x)\tilde{\gamma}_x(1,t) \,\mathrm{d}x.\label{Vdot}
\end{eqnarray}
Let $B = \max_{x\in[0,1]}\bar{p}_1(x)$. The last term of the right hand side of \eqref{Vdot} is estimated as:
\begin{eqnarray}
- \tilde{\gamma}_x(1,t)B\int_0^1\! \hat{\gamma}(x,t) \,\mathrm{d}x &\leq& \frac{1}{2}\int_0^1\! \hat{\gamma}_x(x,t)^2 \,\mathrm{d}x\\
&&+\frac{B^2}{2}\int_0^1\! \tilde{\gamma}_x(x,t)^2 \,\mathrm{d}x.\nonumber
\end{eqnarray}
Applying Poincare's inequality, we have:
\begin{eqnarray}
\dot{U}(t) &\leq& -\frac{1}{2}\int_0^1\! \hat{\gamma}_x(x,t)^2 \,\mathrm{d}x -\left(A-\frac{B^2}{2}\right)\int_0^1\! \tilde{\gamma}_x(x,t)^2 \,\mathrm{d}x\nonumber\\
&\leq&-\frac{1}{8}\int_0^1\! \hat{\gamma}(x,t)^2 \,\mathrm{d}x -\frac{1}{4}\left(A-\frac{B^2}{2}\right)\int_0^1\! \tilde{\gamma}(x,t)^2 \,\mathrm{d}x\nonumber\\
&=&-\frac{1}{4}U(t),
\end{eqnarray}
where we choose $A=B^2$. This shows $\mathbb{L}^2$ exponential stability of the origin for the $\hat{\gamma}$ and $\tilde{\gamma}$ systems. Since the transformation \eqref{trans}-\eqref{errtrans} are invertible with the inverse transformations are defined as follow:
\begin{eqnarray}
\hat{u}(x,t) &=& \hat{\gamma}(x,t) + \int_0^x\! l(x,y)\hat{\gamma}(y,t)\,\mathrm{d}y\label{inv}\\
\tilde{\gamma}(x,t) &=& \tilde{u}(x,t) + \int_x^1\! r(x,y)\tilde{u}(y,t)\,\mathrm{d}y\label{errinv}
\end{eqnarray}
and since $\tilde{\gamma}=\gamma-\hat{\gamma}$, this concludes the proof.
\end{proof}

\section{OUTPUT FEEDBACK STABILIZATION OF SEMILINEAR PARABOLIC SYSTEMS}

We want to show that the linear controller \eqref{cont} works locally for the semilinear system \eqref{main1}. We design the following semilinear observer:
\begin{eqnarray}
\hat{u}_t(x,t) &=& \hat{u}_{xx}(x,t)+f_{NL}(\hat{u}(x,t))+p(x)\tilde{u}(1,t),\label{obsmain1}
\end{eqnarray}
We can write the semilinear parabolic PDEs in a form equivalent up to linear terms to \eqref{obslins1} as follow:
\begin{eqnarray}
\hat{u}_t(x,t) &=& \hat{u}_{xx}(x,t)+\lambda \hat{u}(x,t) + f(\hat{u}(x,t))\nonumber\\
&&+p(x)\tilde{u}(1,t)\label{Ais}
\end{eqnarray}
where
\begin{eqnarray}
f(\hat{u}(x,t)) = f_{NL}(\hat{u}(x,t))-\lambda\hat{u}(x,t),
\end{eqnarray}
and
\begin{eqnarray}
\frac{\partial f_{NL}}{\partial \hat{u}}(0)=\lambda.
\end{eqnarray}
with boundary conditions
\begin{eqnarray}
\hat{u}(0,t) &=& 0,\label{obsmain2}\\
\hat{u}(1,t) &=& U(t).\label{obsmain3}
\end{eqnarray}

\begin{assumption}
\begin{eqnarray}
\lambda\neq0
\end{eqnarray}
\end{assumption}
\begin{remark}
Since $f\in\mathbb{C}^2([0,1])$ and $f(0)=\frac{\partial f}{\partial u}(0)=0$, there exists a $\delta_f$ and positive constants $K_1$, $K_2$, $K_3$ such that if $|u|\leq\delta_f$, then for any $v\in\mathbb{R}$:
\begin{eqnarray}
|f(u)| &\leq& K_1|u|^2\\
\left|\frac{\partial f(u)}{\partial u}\right| &\leq& K_2|u|\\
\left|\frac{\partial^2 f(u)}{\partial u^2}v\right| &\leq& K_3|v|
\end{eqnarray}
\end{remark}

The error system is given by:
\begin{eqnarray}
\tilde{u}_t(x,t) &=& \tilde{u}_{xx}(x,t)+\lambda\tilde{u}(x,t)\nonumber\\
&&+ g(u(x,t),\hat{u}(x,t))-p(x)\tilde{u}(1,t)\label{errnonlin1}
\end{eqnarray}
with boundary conditions
\begin{eqnarray}
\tilde{u}(0,t) &=& 0\label{errnonlin2},\\
\tilde{u}(1,t) &=& 0\label{errnonlin3}.
\end{eqnarray}
The main result of this paper is stated as follows.
\begin{theorem}
Consider system \eqref{main1} with boundary conditions \eqref{main2}-\eqref{main3} and initial condition $u_0\in\mathbb{H}^4([0,1])$ where the kernel $k$ is obtained from \eqref{ker1}-\eqref{ker3}, and system \eqref{obsmain1} with boundary conditions \eqref{obsmain2}-\eqref{obsmain3} and initial condition $\hat{u}_0\in\mathbb{H}^4([0,1])$ where the kernel $p$ is obtained from \eqref{obsker1}-\eqref{obsker3}. Then for every $\mu>0$, there exist $\delta>0$ and $c>0$ such that, if $\|u_0\|_{\mathbb{H}^4}+\|\hat{u}_0\|_{\mathbb{H}^4}\leq\delta$, then:
\begin{eqnarray}
\|u(\cdot,t)\|^2_{\mathbb{H}^4}+\|\hat{u}(\cdot,t)\|^2_{\mathbb{H}^4}\leq ce^{-\mu t}\left(\|u_0\|^2_{\mathbb{H}^4}+\|\hat{u}_0\|^2_{\mathbb{H}^4}\right)
\end{eqnarray}
\end{theorem}

\section{PROOF OF THEOREM 2}

The proof of the Theorem 2 is based on output feedback stabilization and construction of a strict Lyapunov function developed in \cite{Vaz} and \cite{Coron1} for quasilinear hyperbolic systems.

\subsection{Preliminary Definition}

For $\gamma(x,t)\in\mathbb{R}$, we define:
\begin{eqnarray}
\|\gamma\|_{\infty} &=& \sup_{x\in[0,1]}|\gamma(x,t)|,\\
\|\gamma\|_{\mathbb{L}^1} &=& \int_0^1\; |\gamma(x,t)| \,\mathrm{d}x.
\end{eqnarray}
To simplify our notation, we denote $|\gamma|=|\gamma(x,t)|$ and $\|\gamma\|=\|\gamma(\cdot,t)\|$. For $\gamma\in\mathbb{H}^4([0,1])$, recall the following well-known inequalities:
\begin{eqnarray}
\|\gamma\|_{\mathbb{L}^1} &\leq& K_1\|\gamma\|_{\mathbb{L}^2} \leq K_2\|\gamma\|_{\infty},\label{ineq1}\\
\|\gamma\|_{\infty} &\leq& K_3\left(\|\gamma\|_{\mathbb{L}^2}+\|\gamma_x\|_{\mathbb{L}^2}\right) \leq K_4\|\gamma\|_{\mathbb{H}^1},\label{ineq2}\\
\|\gamma_x\|_{\infty} &\leq& K_5\left(\|\gamma_x\|_{\mathbb{L}^2}+\|\gamma_{xx}\|_{\mathbb{L}^2}\right) \leq K_6\|\gamma\|_{\mathbb{H}^2},\label{ineq3}\\
\|\gamma_{xx}\|_{\infty} &\leq& K_7\left(\|\gamma_{xx}\|_{\mathbb{L}^2}+\|\gamma_{xxx}\|_{\mathbb{L}^2}\right) \leq K_8\|\gamma\|_{\mathbb{H}^3},\label{ineq4}\\
\|\gamma_{xxx}\|_{\infty} &\leq& K_9\left(\|\gamma_{xxx}\|_{\mathbb{L}^2}+\|\gamma_{xxxx}\|_{\mathbb{L}^2}\right) \leq K_{10}\|\gamma\|_{\mathbb{H}^4}.\label{ineq5}
\end{eqnarray}
The following functionals are used to simplify the presentation in the upcoming sections.
\begin{eqnarray}
\mathcal{K}[\hat{\gamma}](x) &=& \hat{\gamma}(x,t) -\int_0^x\! k(x,y)\hat{\gamma}(y,t)\,\mathrm{d}y,\\
\mathcal{L}[\hat{\gamma}](x) &=& \hat{\gamma}(x,t) + \int_0^x\! l(x,y)\hat{\gamma}(y,t)\,\mathrm{d}y,\\
\mathcal{R}[\tilde{\gamma}](x) &=& \tilde{\gamma}(x,t) + \int_x^1\! r(x,y)\tilde{\gamma}(y,t)\,\mathrm{d}y,\\
\mathcal{P}[\tilde{\gamma}](x) &=& \tilde{\gamma}(x,t) - \int_x^1\! p(x,y)\tilde{\gamma}(y,t)\,\mathrm{d}y.
\end{eqnarray}
Since the control direct \eqref{trans} and \eqref{errtrans}, and the inverse kernels \eqref{inv} and \eqref{errinv}, are $\mathbb{C}^2(\mathcal{T})$, these functionals satisfy the following bounds:
\begin{eqnarray}
\left|\mathcal{K}[\hat{\gamma}]\right| &\leq& K_1\left(|\hat{\gamma}|+\|\hat{\gamma}\|_{\mathbb{L}^1}\right),\\
\left|\mathcal{L}[\hat{\gamma}]\right| &\leq& K_2\left(|\hat{\gamma}|+\|\hat{\gamma}\|_{\mathbb{L}^1}\right),\\
\left|\mathcal{R}[\tilde{\gamma}]\right| &\leq& K_3\left(|\tilde{\gamma}|+\|\tilde{\gamma}\|_{\mathbb{L}^1}\right),\\
\left|\mathcal{P}[\tilde{\gamma}]\right| &\leq& K_4\left(|\tilde{\gamma}|+\|\tilde{\gamma}\|_{\mathbb{L}^1}\right).
\end{eqnarray}

It can be proven that the transformations \eqref{trans}, \eqref{errtrans}, \eqref{inv}, and \eqref{errinv} map system:
\begin{eqnarray}
\hat{\gamma}_t(x,t) &=& \hat{\gamma}_{xx}(x,t)+F[\hat{\gamma}]-\bar{p}_1(x)\tilde{\gamma}_x(1,t),\label{er1}\\
\hat{\gamma}(0,t) &=& 0,\\
\hat{\gamma}(1,t) &=& 0,\label{er3}\\
\tilde{\gamma}_t(x,t) &=& \tilde{\gamma}_{xx}(x,t)+G[\hat{\gamma},\tilde{\gamma}],\label{li1}\\
\tilde{\gamma}(0,t) &=& 0,\\
\tilde{\gamma}(1,t) &=& 0,\label{li3}
\end{eqnarray}
into \eqref{Ais}-\eqref{obsmain3} and \eqref{errnonlin1}-\eqref{errnonlin3}, where:
\begin{eqnarray}
F[\hat{\gamma}] &=& \mathcal{K}[f(\mathcal{L}[\hat{\gamma}])],\\
G[\hat{\gamma},\tilde{\gamma}] &=& \mathcal{R}[g(\mathcal{L}[\hat{\gamma}]),g(\mathcal{P}[\tilde{\gamma}])].
\end{eqnarray}
The above functionals satisfy the following bounds:
\begin{eqnarray}
\left|F[\hat{\gamma}]\right| &\leq& K_1\left(|\hat{\gamma}|^2+\|\hat{\gamma}\|_{\mathbb{L}^2}^2\right),\\
\left|G[\hat{\gamma},\tilde{\gamma}]\right| &\leq& K_2\left(|\hat{\gamma}|^2+\|\hat{\gamma}\|_{\mathbb{L}^2}^2\right)+K_3\left(|\tilde{\gamma}|^2+\|\tilde{\gamma}\|_{\mathbb{L}^2}^2\right).
\end{eqnarray}

The $\mathbb{H}^4$ local stability for $\gamma$ and $\hat{\gamma}$ are proved by relating the growth of $\|\hat{\gamma}\|_{\mathbb{L}^2}+\|\tilde{\gamma}\|_{\mathbb{L}^2}$, $\|\hat{\gamma}_t\|_{\mathbb{L}^2}+\|\tilde{\gamma}_t\|_{\mathbb{L}^2}$, and $\|\hat{\gamma}_{tt}\|_{\mathbb{L}^2}+\|\tilde{\gamma}_{tt}\|_{\mathbb{L}^2}$ with $\|\hat{\gamma}\|_{\mathbb{H}^4}+\|\tilde{\gamma}\|_{\mathbb{H}^4}$. The relations of these norms are given in the following lemmas.

\begin{lemma}
There exists $\delta>0$ such that for $\|\hat{\gamma}\|_{\infty}+\|\tilde{\gamma}\|_{\infty}<0$, then the norm defined by $\|\hat{\gamma}_t\|_{\mathbb{L}^2}+\|\tilde{\gamma}_t\|_{\mathbb{L}^2}+\|\hat{\gamma}\|_{\mathbb{L}^2}+\|\tilde{\gamma}\|_{\mathbb{L}^2}$ is equivalent to $\|\hat{\gamma}\|_{\mathbb{H}^2}+\|\tilde{\gamma}\|_{\mathbb{H}^2}$.
\end{lemma}
\begin{proof}
The lemma is proven by bounding the norms in \eqref{er1} and \eqref{li1} for small $\|\hat{\gamma}\|_{\infty}+\|\tilde{\gamma}\|_{\infty}$.
\end{proof}
\begin{lemma}
There exists $\delta>0$ such that for $\|\hat{\gamma}\|_{\infty}+\|\tilde{\gamma}\|_{\infty}+\|\hat{\gamma}_t\|_{\infty}+\|\tilde{\gamma}_t\|_{\infty}<0$, then the norm defined by $\|\hat{\gamma}_{tt}\|_{\mathbb{L}^2}+\|\tilde{\gamma}_{tt}\|_{\mathbb{L}^2}+\|\hat{\gamma}_t\|_{\mathbb{L}^2}+\|\tilde{\gamma}_t\|_{\mathbb{L}^2}+\|\hat{\gamma}\|_{\mathbb{L}^2}+\|\tilde{\gamma}\|_{\mathbb{L}^2}$ is equivalent to $\|\hat{\gamma}\|_{\mathbb{H}^4}+\|\tilde{\gamma}\|_{\mathbb{H}^4}$.
\end{lemma}
\begin{proof}
Takin two $x$-derivatives in \eqref{er1} and \eqref{li1}, bounding these variables and using lemma 1, we conclude the proof.
\end{proof}

\subsection{Analyzing the Growth of $\|\hat{\gamma}\|_{\mathbb{L}^2}+\|\tilde{\gamma}\|_{\mathbb{L}^2}$}

The first derivative of \eqref{U} with respect to $t$ along \eqref{er1}-\eqref{li3} is given by:
\begin{eqnarray}
\dot{U}(t) &=& -\int_0^1\! \hat{\gamma}_x(x,t)^2 \,\mathrm{d}x -A\int_0^1\! \tilde{\gamma}_x(x,t)^2 \,\mathrm{d}x \nonumber\\
&&- \int_0^1\! \hat{\gamma}(x,t)\bar{p}_1(x)\tilde{\gamma}_x(1,t) \,\mathrm{d}x\label{Vdotn}\\
&&+\int_0^1\! \hat{\gamma}(x,t)F[\hat{\gamma}] \,\mathrm{d}x+\int_0^1\! \tilde{\gamma}(x,t)G[\hat{\gamma},\tilde{\gamma}] \,\mathrm{d}x.\nonumber
\end{eqnarray}
The last terms of the right hand side of \eqref{Vdotn} are estimated as follow:
\begin{eqnarray}
\int_0^1\! |\hat{\gamma}||F[\hat{\gamma}]| \,\mathrm{d}x&\leq& K_1\|\hat{\gamma}\|_{\infty}\|\hat{\gamma}\|^2_{\mathbb{L}^2},\\
\int_0^1\! |\tilde{\gamma}||G[\hat{\gamma},\tilde{\gamma}]| \,\mathrm{d}x&\leq& K_2\|\tilde{\gamma}\|_{\infty}\left(\|\hat{\gamma}\|^2_{\mathbb{L}^2}+\|\tilde{\gamma}\|^2_{\mathbb{L}^2}\right).
\end{eqnarray}
Hence, we have:
\begin{eqnarray}
&&\int_0^1\! |\hat{\gamma}||F[\hat{\gamma}]| \,\mathrm{d}x+\int_0^1\! |\tilde{\gamma}||G[\hat{\gamma},\tilde{\gamma}]| \,\mathrm{d}x\leq\nonumber\\
&& K_3\left(\|\hat{\gamma}\|_{\infty}+\|\tilde{\gamma}\|_{\infty}\right)U
\end{eqnarray}
Applying \eqref{ineq2}, we have:
\begin{eqnarray}
&&\int_0^1\! |\hat{\gamma}||F[\hat{\gamma}]| + |\tilde{\gamma}||G[\hat{\gamma},\tilde{\gamma}]| \,\mathrm{d}x\leq \nonumber\\
&&K_4\left(\|\hat{\gamma}_x\|_{\infty}+\|\tilde{\gamma}_x\|_{\infty}\right)U+K_5U^{3/2}
\end{eqnarray}
Thus, we have the following result.

\begin{theorem}
There exists $\delta_1$ such that if $\|\gamma\|_{\infty}<\delta_1$ then:
\begin{eqnarray}
\dot{U} &\leq& -\mu_1U + C_1\left(\|\hat{\gamma}_x\|_{\infty}+\|\tilde{\gamma}_x\|_{\infty}\right)U+ C_2 U^{3/2},
\end{eqnarray}
where $\mu_1$, $C_1$, and $C_2$ are positive constants.
\end{theorem}

\subsection{Analyzing the Growth of $\|\hat{\gamma}_t\|_{\mathbb{L}^2}+\|\tilde{\gamma}_t\|_{\mathbb{L}^2}$}

Define $\hat{\eta} = \hat{\gamma}_t$ and $\tilde{\eta} = \tilde{\gamma}_t$. Remark that the norms $\hat{\eta}$, $\tilde{\eta}$ and $\hat{\gamma}_{xx}$, $\tilde{\gamma}_{xx}$ are related according to lemma 1. Taking a partial derivative in $t$ along \eqref{er1}-\eqref{li3}, we have:
\begin{eqnarray}
\hat{\eta}_t(x,t) &=& \hat{\eta}_{xx}(x,t)+F_1[\hat{\gamma},\hat{\eta}]-\bar{p}_1(x)\tilde{\eta}_x(1,t),\label{err1}\\
\hat{\eta}(0,t) &=& 0,\\
\hat{\eta}(1,t) &=& 0,\\
\tilde{\eta}_t(x,t) &=& \tilde{\eta}_{xx}(x,t)+G_1[\hat{\gamma},\tilde{\gamma},\hat{\eta},\tilde{\eta}],\\
\tilde{\eta}(0,t) &=& 0,\\
\tilde{\eta}(1,t) &=& 0,\label{err3}
\end{eqnarray}
where
\begin{eqnarray}
F_1[\hat{\gamma},\hat{\eta}] &=& \mathcal{K}\left[\frac{\partial f}{\partial\hat{\gamma}}(\mathcal{L}[\hat{\gamma}])\mathcal{L}[\hat{\eta}]\right],\\
G_1[\hat{\gamma},\tilde{\gamma},\hat{\eta},\tilde{\eta}] &=& \mathcal{R}\left[\frac{\partial g}{\partial\hat{\gamma}}(\mathcal{L}[\hat{\gamma}])\mathcal{L}[\hat{\eta}],\frac{\partial g}{\partial\tilde{\gamma}}(\mathcal{P}[\tilde{\gamma}])\mathcal{P}[\tilde{\eta}]\right].
\end{eqnarray}
The above functionals satisfy the following bounds:
\begin{eqnarray}
\left|F_1[\hat{\gamma},\hat{\eta}]\right| &\leq& K_1 \left(|\hat{\gamma}|+\|\hat{\gamma}\|_{\mathbb{L}^1}\right)\left(|\hat{\eta}|+\|\hat{\eta}\|_{\mathbb{L}^1}\right),\\
\left|G_1[\hat{\gamma},\tilde{\gamma},\hat{\eta},\tilde{\eta}]\right| &\leq& K_2\left(|\hat{\gamma}|+\|\hat{\gamma}\|_{\mathbb{L}^1}\right)\left(|\hat{\eta}|+\|\hat{\eta}\|_{\mathbb{L}^1}\right)\\
&&+K_3\left(|\tilde{\gamma}|+\|\tilde{\gamma}\|_{\mathbb{L}^1}\right)\left(|\tilde{\eta}|+\|\tilde{\eta}\|_{\mathbb{L}^1}\right).\nonumber
\end{eqnarray}
Let
\begin{eqnarray}
V(t) = \frac{1}{2}\int_0^1\! \hat{\eta}(x,t)^2\,\mathrm{d}x + \frac{A}{2}\int_0^1\! \tilde{\eta}(x,t)^2\,\mathrm{d}x.\label{V}
\end{eqnarray}
The first derivative of \eqref{V} with respect to $t$ along \eqref{err1}-\eqref{err3} is given by:
\begin{eqnarray}
\dot{V}(t) &=& -\int_0^1\! \hat{\eta}_x(x,t)^2 \,\mathrm{d}x -A\int_0^1\! \tilde{\eta}_x(x,t)^2 \,\mathrm{d}x \nonumber\\
&&- \int_0^1\! \hat{\eta}(x,t)\bar{p}_1(x)\tilde{\eta}_x(1,t) \,\mathrm{d}x\nonumber\\
&&+\int_0^1\! \hat{\eta}(x,t)F_1[\hat{\gamma},\hat{\eta}] \,\mathrm{d}x\nonumber\\
&&+\int_0^1\! \tilde{\eta}(x,t)G_1[\hat{\gamma},\tilde{\gamma},\hat{\eta},\tilde{\eta}] \,\mathrm{d}x.\label{Vdotnn}
\end{eqnarray}
The last terms of the right hand side of \eqref{Vdotnn} are estimated as follow:
\begin{eqnarray}
\int_0^1\! |\hat{\eta}||F_1[\hat{\gamma},\hat{\eta}]| \,\mathrm{d}x&\leq& K_4\|\hat{\eta}\|_{\infty}\left(\|\hat{\gamma}\|^2_{\mathbb{L}^2}+\|\hat{\eta}\|^2_{\mathbb{L}^2}\right),\nonumber\\
\\
\int_0^1\! |\tilde{\eta}||G_1[\hat{\gamma},\tilde{\gamma},\hat{\eta},\tilde{\eta}]| \,\mathrm{d}x &\leq& K_5\|\tilde{\eta}\|_{\infty}\left(\|\hat{\gamma}\|^2_{\mathbb{L}^2}+\|\hat{\eta}\|^2_{\mathbb{L}^2}\right.\nonumber\\
&&\left.+\|\tilde{\gamma}\|^2_{\mathbb{L}^2}+\|\tilde{\eta}\|^2_{\mathbb{L}^2}\right).
\end{eqnarray}
Hence, we have
\begin{eqnarray}
&&\int_0^1\! |\hat{\eta}||F_1[\hat{\gamma},\hat{\eta}]|+|\tilde{\eta}||G_1[\hat{\gamma},\tilde{\gamma},\hat{\eta},\tilde{\eta}]| \,\mathrm{d}x\leq\nonumber\\
&& K_6\left(\|\hat{\eta}\|_{\infty}+\|\tilde{\eta}\|_{\infty}\right)V.
\end{eqnarray}
Since
\begin{eqnarray}
\|\hat{\eta}\|_{\infty}+\|\tilde{\eta}\|_{\infty} &\leq& K_7\left(\|\|\hat{\eta}_x\|_{\infty}+\|\tilde{\eta}_x\|_{\infty}\right).
\end{eqnarray}
We have:
\begin{eqnarray}
&&\int_0^1\! |\hat{\eta}||F_3[\hat{\gamma},\hat{\eta}]|+|\tilde{\eta}||F_4[\hat{\gamma},\tilde{\gamma},\hat{\eta},\tilde{\eta}]| \,\mathrm{d}x\leq\nonumber\\
&& K_8\left(\|\hat{\eta}_x\|_{\infty}+\|\tilde{\eta}_x\|_{\infty}\right)V\nonumber\\
\end{eqnarray}
Thus, we have the following result.
\begin{theorem}
There exists $\delta_2$ such that if $\|\hat{\gamma}\|_{\infty}+\|\tilde{\gamma}\|_{\infty}<\delta_2$ then
\begin{eqnarray}
\dot{V}(t) &\leq& -\mu_2V(t) + C_3 \left(\|\hat{\eta}_x\|_{\infty}+\|\tilde{\eta}_x\|_{\infty}\right)V(t)
\end{eqnarray}
where $\mu_2$ and $C_3$ are positive constants.
\end{theorem}

\subsection{Analyzing the Growth of $\|\hat{\gamma}_{tt}\|_{\mathbb{L}^2}+\|\tilde{\gamma}_{tt}\|_{\mathbb{L}^2}$}

Define $\hat{\theta}=\hat{\eta}_t$ and $\tilde{\theta}=\tilde{\eta}_t$. Remark that the norms $\hat{\theta}$, $\tilde{\theta}$ and $\hat{\gamma}_{xxxx}$, $\tilde{\gamma}_{xxxx}$ are related according to lemma 2. Taking a partial derivative in $t$ along \eqref{err1}-\eqref{err3}, we have:
\begin{eqnarray}
\hat{\theta}_t(x,t) &=& \hat{\theta}_{xx}(x,t)+F_2[\hat{\gamma},\hat{\eta},\hat{\theta}]-\bar{p}_1(x)\tilde{\theta}_x(1,t),\\
\hat{\theta}(0,t) &=& 0,\\
\hat{\theta}(1,t) &=& 0,\\
\tilde{\theta}_t(x,t) &=& \tilde{\theta}_{xx}(x,t)+G_2[\hat{\gamma},\tilde{\gamma},\hat{\eta},\tilde{\eta},\hat{\theta},\tilde{\theta}],\\
\tilde{\theta}(0,t) &=& 0,\\
\tilde{\theta}(1,t) &=& 0,
\end{eqnarray}
where
\begin{eqnarray}
&&F_2[\hat{\gamma},\hat{\eta},\hat{\theta}] = \nonumber\\ &&\mathcal{K}\left[\frac{\partial^2f}{\partial\hat{\gamma}^2}(\mathcal{L}[\hat{\gamma}])\mathcal{L}[\hat{\eta}]\mathcal{L}[\hat{\eta}]+\frac{\partial f}{\partial\hat{\gamma}}(\mathcal{L}[\hat{\gamma}])\mathcal{L}[\hat{\theta}]\right]\\
&&G_2[\hat{\gamma},\tilde{\gamma},\hat{\eta},\tilde{\eta},\hat{\theta},\tilde{\theta}] = \nonumber\\ &&\mathcal{R}\left[\frac{\partial^2g}{\partial\hat{\gamma}^2}(\mathcal{L}[\hat{\gamma}])\mathcal{L}[\hat{\eta}]\mathcal{L}[\hat{\eta}]+\frac{\partial g}{\partial\hat{\gamma}}(\mathcal{L}[\hat{\gamma}])\mathcal{L}[\hat{\theta}],\right.\nonumber\\
&&\left.\frac{\partial^2g}{\partial\tilde{\gamma}^2}(\mathcal{P}[\tilde{\gamma}])\mathcal{P}[\tilde{\eta}]\mathcal{P}[\tilde{\eta}]+\frac{\partial g}{\partial\tilde{\gamma}}(\mathcal{P}[\tilde{\gamma}])\mathcal{P}[\tilde{\theta}]\right]
\end{eqnarray}
The above functionals satisfy the following bounds:
\begin{eqnarray}
&&|F_2[\hat{\gamma},\hat{\eta},\hat{\theta}]| \leq K_1 \left(|\hat{\eta}|^2+\|\hat{\eta}\|_{\mathbb{L}^1}^2\right)\nonumber\\
&&+ K_2 \left(|\hat{\gamma}|+\|\hat{\gamma}\|_{\mathbb{L}^1}\right)\times\left(|\hat{\theta}|+\|\hat{\theta}\|_{\mathbb{L}^1}\right)\\
&&|G_2[\hat{\gamma},\tilde{\gamma},\hat{\eta},\tilde{\eta},\hat{\theta},\tilde{\theta}]| \leq K_3 \left(|\hat{\eta}|^2+\|\hat{\eta}\|_{\mathbb{L}^1}^2\right)\nonumber\\
&&+K_4 \left(|\tilde{\eta}|^2+\|\tilde{\eta}\|_{\mathbb{L}^1}^2\right) \nonumber\\
&&+ K_5 \left(|\hat{\gamma}|+\|\hat{\gamma}\|_{\mathbb{L}^1}\right)\left(|\hat{\theta}|+\|\hat{\theta}\|_{\mathbb{L}^1}\right)\nonumber\\
&&+ K_6 \left(|\tilde{\gamma}|+\|\tilde{\gamma}\|_{\mathbb{L}^1}\right)\left(|\tilde{\theta}|+\|\tilde{\theta}\|_{\mathbb{L}^1}\right)
\end{eqnarray}
Let
\begin{eqnarray}
W(t) = \frac{1}{2}\int_0^1\! \hat{\theta}(x,t)^2\,\mathrm{d}x + \frac{A}{2}\int_0^1\! \tilde{\theta}(x,t)^2\,\mathrm{d}x.\label{W}
\end{eqnarray}
The first derivative of \eqref{W} with respect to $t$ along \eqref{err1}-\eqref{err3} is given by:
\begin{eqnarray}
\dot{W}(t) &=& -\int_0^1\! \hat{\theta}_x(x,t)^2 \,\mathrm{d}x -A\int_0^1\! \tilde{\theta}_x(x,t)^2 \,\mathrm{d}x \nonumber\\
&&- \int_0^1\! \hat{\theta}(x,t)\bar{p}_1(x)\tilde{\theta}_x(1,t) \,\mathrm{d}x\nonumber\\
&&+\int_0^1\! \hat{\theta}(x,t)F_2[\hat{\gamma},\hat{\eta},\hat{\theta}] \,\mathrm{d}x\nonumber\\
&&+\int_0^1\! \tilde{\theta}(x,t)G_2[\hat{\gamma},\tilde{\gamma},\hat{\eta},\tilde{\eta},\hat{\theta},\tilde{\theta}] \,\mathrm{d}x.\label{Vdotnnn}
\end{eqnarray}
The last terms of the right hand side of \eqref{Vdotnnn} are estimated as follow:
\begin{eqnarray}
&&\int_0^1\! |\hat{\theta}||F_2[\hat{\gamma},\hat{\eta},\hat{\theta}]| \,\mathrm{d}x\leq\nonumber\\
&&K_7\|\hat{\theta}\|_{\infty}\left(\|\hat{\gamma}\|^2_{\mathbb{L}^2}+\|\hat{\eta}\|^2_{\mathbb{L}^2}+\|\hat{\theta}\|^2_{\mathbb{L}^2}\right),\\
&&\int_0^1\! |\tilde{\theta}||G_2[\hat{\gamma},\tilde{\gamma},\hat{\eta},\tilde{\eta},\hat{\theta},\tilde{\theta}]| \,\mathrm{d}x \leq\nonumber\\
&&K_8\|\tilde{\theta}\|_{\infty}\left(\|\hat{\gamma}\|^2_{\mathbb{L}^2}+\|\hat{\eta}\|^2_{\mathbb{L}^2}+\|\hat{\theta}\|^2_{\mathbb{L}^2}\right.\nonumber\\
&&\left.+\|\tilde{\gamma}\|^2_{\mathbb{L}^2}+\|\tilde{\eta}\|^2_{\mathbb{L}^2}+\|\tilde{\theta}\|^2_{\mathbb{L}^2}\right).
\end{eqnarray}
Hence, we have:
\begin{eqnarray}
&&\int_0^1\! |\hat{\theta}||F_2[\hat{\gamma},\hat{\eta},\hat{\theta}]|+|\tilde{\theta}||G_2[\hat{\gamma},\tilde{\gamma},\hat{\eta},\tilde{\eta},\hat{\theta},\tilde{\theta}]| \,\mathrm{d}x\nonumber\\
&\leq& K_9\left(\|\hat{\theta}\|_{\infty}+\|\tilde{\theta}\|_{\infty}\right)\left(U+V+W\right)
\end{eqnarray}
Relating $\|\hat{\theta}\|_{\infty}$, $\|\tilde{\theta}\|_{\infty}$ and $\|\hat{\eta}_{xx}\|_{\infty}$, $\|\tilde{\eta}_{xx}\|_{\infty}$, we have:
\begin{eqnarray}
&&\int_0^1\! |\hat{\theta}||F_2[\hat{\gamma},\hat{\eta},\hat{\theta}]|+|\tilde{\theta}||G_2[\hat{\gamma},\tilde{\gamma},\hat{\eta},\tilde{\eta},\hat{\theta},\tilde{\theta}]| \,\mathrm{d}x\leq\nonumber\\
&& K_{10}\left(WV^{1/2}+VW^{1/2}+W^{3/2}\right).
\end{eqnarray}
Thus, we have the following result.
\begin{theorem}
There exists $\delta_3$ such that if $\|\hat{\gamma}\|_{\infty}+\|\tilde{\gamma}\|_{\infty}+\|\hat{\eta}\|_{\infty}+\|\tilde{\eta}\|_{\infty}<\delta_3$ then
\begin{eqnarray}
\dot{W} &\leq&-\frac{1}{4}Z+C_4\left(WZ^{1/2}+VW^{1/2}+W^{3/2}\right)
\end{eqnarray}
where $\mu_3$ and $C_4$ are positive constants.
\end{theorem}

\subsection{Proof of $\mathbb{H}^4$ Stability of ($\hat{\gamma}$, $\tilde{\gamma}$)}

Define $S=U+V+W$, and combining theorem 3, 4, and 5, there exists $\delta$ such that if $\|\hat{\gamma}\|_{\infty}+\|\tilde{\gamma}\|_{\infty}+\|\hat{\eta}\|_{\infty}+\|\tilde{\eta}\|_{\infty}<\delta$ then
\begin{eqnarray}
\dot{S}\leq-\mu S+CS^{3/2}
\end{eqnarray}
for some positive $\mu$ and $C$. Then, for any $\mu_0$ such that $0<\mu_0<\mu$, there exists $\delta_0$ such that
\begin{eqnarray}
C\left|S^{3/2}\right| < \left(\mu-\mu_0\right)S,\;\forall S<\delta_0,
\end{eqnarray}
which implies that
\begin{eqnarray}
\dot{S}<-\mu_0S,\;\forall S<\delta_0.
\end{eqnarray}
Noting that $\|\hat{\gamma}\|_{\infty}+\|\tilde{\gamma}\|_{\infty}+\|\hat{\eta}\|_{\infty}+\|\tilde{\eta}\|_{\infty}\leq \bar{C}S$ for $\bar{C}>0$, then for sufficiently small $S(0)$, we have $S(t)\rightarrow0$ exponentially. Since $S$ is equivalent to $\|\hat{\gamma}\|_{\mathbb{H}^4}+\|\tilde{\gamma}\|_{\mathbb{H}^4}$ when $\|\hat{\gamma}\|_{\infty}+\|\tilde{\gamma}\|_{\infty}+\|\hat{\eta}\|_{\infty}+\|\tilde{\eta}\|_{\infty}$ is sufficiently small, this concludes the proof.

\section{NUMERICAL EXAMPLE}

We consider output feedback boundary stabilization problem of the FitzHugh-Nagumo equation as follow:
\begin{eqnarray}
u_t &=& u_{xx}-u(1-u)^2,\label{fish}\\
u(0,t) &=& 0,\\
u(1,t) &=& U(t).
\end{eqnarray}
This equation was proposed independently by FitzHugh \cite{Fitz} and Nagumo, et al. \cite{Nagumo} during 60's to model active pulse transmission line in nerve membrane. Applying our control law \eqref{cont}, where the state $\hat{u}$ is generated from \eqref{obsmain1}, the state is derived to its equilibrium $\hat{u}\equiv0$, as can be seen from \textbf{Fig}. \ref{fig1}.
\begin{figure}[h!]
  \centering
      \includegraphics[width=0.5\textwidth]{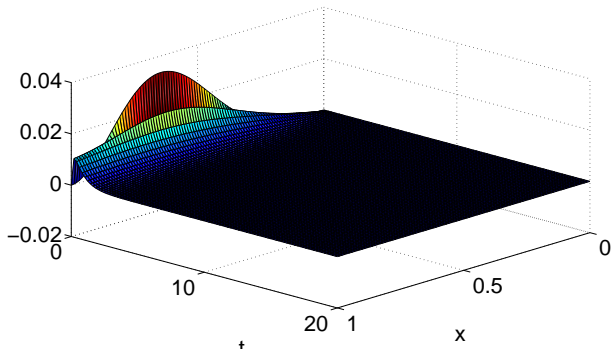}
  \caption{Stabilization by output feedback.}
\label{fig1}
\end{figure}

The estimation error also converge to zero as shown from \textbf{Fig}. \ref{fig2}.
\begin{figure}[h!]
  \centering
      \includegraphics[width=0.5\textwidth]{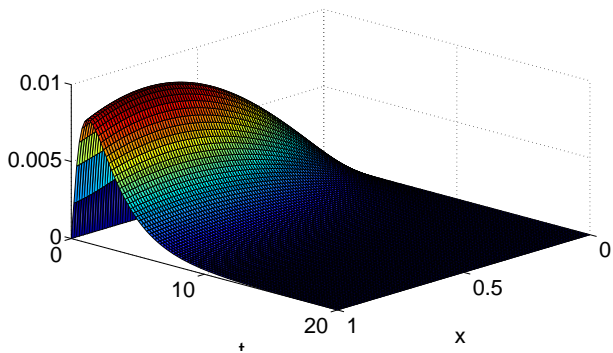}
  \caption{Estimation error $\tilde{u}(x,t)$.}
\label{fig2}
\end{figure}

\section{CONCLUSIONS AND FUTURE WORKS}

We have solved output feedback boundary stabilization problem for a class of semilinear parabolic PDEs with actuation and measurement on only one boundary (collocated setup). The state and the observer error systems are shown to be locally exponentially stable in the $\mathbb{H}^4$ norm. The strict Lyapunov function used in this paper could be used to handle nonlinearity in other PDEs such as the Korteweg-de Vries equation and the Burgers equation. We aim to address these problems in future work.

\end{document}